\documentclass[a4paper, svgnames, 11pt]{amsart}
\usepackage[english]{babel}
 \usepackage{amsmath, amsthm, amssymb, color}
\usepackage[pagebackref,colorlinks,citecolor=red,urlcolor=blue,
bookmarks=false,hypertexnames=true]{hyperref}
\usepackage[capitalize,noabbrev]{cleveref}
\usepackage{graphicx}
\usepackage[dvipsnames]{xcolor}
\usepackage{caption}
\usepackage{subcaption}
\usepackage{mathtools}
\usepackage{enumerate}
\usepackage{verbatim}
\usepackage{tikz}
\usepackage{tikz-cd}
\usepackage[inline,shortlabels]{enumitem}
\usepackage{hyperref}
\hypersetup{
 colorlinks=true,
 linkcolor=blue,
 filecolor=blue, 
 urlcolor=blue,
 citecolor=blue,
}
\usepackage{url}
\usepackage{longtable}
\usepackage{bm}
\usepackage{soul}
\usepackage{stmaryrd}
\usepackage{tabularray}
\usepackage{hhline}
\usepackage[a4paper]{geometry}
\numberwithin{equation}{section}

\newtheorem{thm}{Theorem}

\numberwithin{thm}{section}

\newtheorem*{thm*}{Theorem}
\newtheorem*{con*}{Conjecture}

\newtheorem{lemma}[thm]{Lemma}

\theoremstyle{definition}

\newtheorem{remark}[thm]{Remark}

\DeclareMathOperator{\Aut}{Aut}
\DeclareMathOperator{\GL}{GL}

\DeclareMathOperator{\tuC}{C}
\DeclareMathOperator{\tuD}{D}
\DeclareMathOperator{\tuB}{B}
\DeclareMathOperator{\tuBd}{B^\bullet}

\DeclareMathOperator{\GP}{G}

\DeclareMathOperator{\Mat}{Mat}

\DeclareMathOperator{\ZG}{Z}
\DeclareMathOperator{\rk}{rk}
\DeclareMathOperator{\irk}{irk}

\DeclareMathOperator{\Rad}{rad}

\DeclareMathOperator{\irr}{irr}

\DeclareMathOperator{\degprim}{degprim}
\DeclareMathOperator{\prim}{prim}

\newcommand{\N}{\mathbb{N}}

\newcommand{\Q}{\mathbb{Q}}

\newcommand{\F}{\mathbb{F}}
\newcommand{\Lee}{\mathcal{L}}
\newcommand{\Z}{\mathbb{Z}}

\renewcommand{\phi}{\varphi}
\renewcommand{\leq}{\leqslant}
\renewcommand{\geq}{\geqslant}
\renewcommand{\epsilon}{\varepsilon}
\renewcommand{\P}{\mathbb{P}}
\newcommand{\alg}{K}
\newcommand{\ff}{\mathbb{F}_p}
\newcommand{\varaff}[1]{\mathcal{V}_{\mathrm{aff}}(#1)}
\newcommand{\varproj}[1]{\mathcal{V}_{\mathrm{proj}}(#1)}
\newcommand{\gen}[1]{\langle #1 \rangle}

\newcommand{\cor}[1]{\mathcal{#1}}

\makeatletter
\renewcommand\subsection{\@startsection{subsection}{2}%
	\z@{.5\linespacing\@plus.7\linespacing}{-.5em}%
	{\normalfont \bfseries}}
\makeatother

\makeatletter
\@namedef{subjclassname@2020}{%
	\textup{2020} Mathematics Subject Classification}
\makeatother

\begin{document}

\title[Geometric invariants for $p$-groups of class 2 and exponent $p$]{
Geometric invariants for $p$-groups \\ of class 2 and exponent $p$}

\author[E.\ A.\ O'Brien]{E.\ A.\ O'Brien}
\address[E.\ A.\ O'Brien]{
Department of Mathematics,  
University of Auckland,
Private Bag 92019,
Auckland, 
New Zealand
}
\email{e.obrien@auckland.ac.nz}

\author[Mima Stanojkovski]{Mima Stanojkovski}
\address[Mima Stanojkovski]{
Dipartimento di Matematica, Universit\`a di Trento
}
\email{mima.stanojkovski@unitn.it}

\makeatletter
\@namedef{subjclassname@2020}{
 \textup{2020} Mathematics Subject Classification}
\makeatother


\thanks{
O’Brien was supported by the Marsden Fund of New Zealand Grant 23-UOA-080
and by a Research Award of the Alexander von Humboldt Foundation.
Stanojkovski is a member of the Indam
group GNSAGA and is funded by the Italian program Rita Levi
Montalcini, edition 2020. We thank Bettina Eick, 
Fulvio Gesumundo, Josh Maglione, 
M.F.\ Newman, and  Christopher Voll for helpful discussions.
We thank the referee for helpful guidance on the content and exposition. 
}

\begin{abstract}
We introduce geometric invariants for 
$p$-groups of class $2$ and exponent $p$.
We report on their effectiveness in distinguishing 
among 5-generator $p$-groups of this type.
\end{abstract}

\maketitle

\setcounter{tocdepth}{2}

\section{Introduction}

The problem of deciding whether two 
presentations present the same group was 
formulated by Dehn \cite{Dehn11}. 
Adian \cite{Adian58} and Rabin \cite{Rabin58} showed that the 
isomorphism problem for finitely presented groups is undecidable.

A typically easier task is to establish that two groups are 
non-isomorphic by exhibiting 
distinguishing invariants.  This approach is particularly 
useful in providing independent evidence 
that an existing list of groups  -- perhaps obtained theoretically --
is irredundant.  But families of finite $p$-groups,
where the prime $p$ is a parameter in the group descriptions,
pose particular challenges. 
It is difficult to find ``natural" invariants which distinguish among 
very similar groups -- and it may be impossible
to compute their values for moderate values of the prime. 
For example, power maps are sometimes used
to distinguish among the 267 groups of order $2^6$. 

To set a context for our work, we review the algorithm 
of O'Brien \cite{OBrien94} 
in a special case: decide 
isomorphism between two $n$-generator groups, 
$G$ and $H$, of order $p^{n + d}$, nilpotency class 2 and exponent $p$.
The Frattini quotient of $G$ coincides with that of $H$,
and the automorphism group of this quotient is $\GL(V)$ where $V=\F_p^n$.  
We write down the action of $\GL(V)$ on the exterior 
square $\Lambda^2 (V)$, a space of dimension $n (n - 1)/ 2$.
To each of $G$ and $H$ we associate 
a $d$-dimensional subspace of $\Lambda^2 (V)$.
Now $G \cong H$ if and only if
the duals of these two subspaces are in the same $\GL(V)$-orbit.
The length of this orbit 
may prevent its construction. 
In summary, a complete solution to the isomorphism
problem for this family of groups depends on knowledge
of the $\GL(V)$-orbits of $d$-dimensional spaces of $\Lambda^2(V)$.
Sun \cite{Sun23} recently announced a new
algorithm with running time $m^{O((\log m)^{5/6})}$ 
to decide isomorphism between two such groups of order $m$.

Standard group-theoretic invariants for $p$-groups of class 2 and 
exponent $p$ often coincide. 
 We introduce invariants from algebraic geometry
which apply specifically to these groups.
The invariants are extracted from the skew-symmetric matrices of linear forms
that these groups determine.
We demonstrate that the invariants are discriminating -- 
they sometimes partition a family
up to isomorphism -- and they can be computed 
for a range of primes.  Our invariants are computed in
affine varieties, rather than in the $p$-group as for 
many existing invariants, and successful computations 
seem less dependent on the size of the prime. 
We plan to investigate the effectiveness of these novel 
invariants for other related questions. 

The structure of the paper is the following. 
In ~\cref{sec:gps-from-tensors}
we summarise the standard Baer correspondence 
between $p$-groups  of class $2$  and exponent $p$ 
and skew-symmetric  matrices of linear forms over $\ff$. 
A consequence is 
that the rank loci of a skew-symmetric  matrix of linear  forms over
$\F_p$ are determined  by the isomorphism class  of the corresponding
group. 
\cref{th:invariants} lists group invariants attached  to the rank ideals.  
In Section \ref{sec:4-3} we illustrate how these 
distinguish up to isomorphism the $4$-generator groups of order $p^7$, 
class 2, and exponent $p$.  In \cref{sec:5-gen}, 
building on work of Brahana \cite{Brah51, Brah58}, Copetti \cite{Cop04}, 
and Vaughan-Lee \cite{VL15}, we explore the corresponding 
$5$-generator groups of order dividing $p^9$.  
We distinguish these with one exception: two groups of 
order $p^8$.  In \cref{thm:lee} we show that the examples of \cite{Lee16} 
motivated by Higman's PORC conjecture have minimal order.   
An explicit discussion of the 
well-known connection with tensors appears in~\cite{OBS24}.


\section{Geometric invariants associated to groups}\label{Group-section}
Let $\alg[x_1,\ldots,x_d]$ be the
polynomial ring over a field $\alg$ in the indeterminates
$x_1,\ldots,x_d$. 
We write ${\bm x}=(x_1,\ldots,x_d)$ and
$\alg[{\bm{x}}]=\alg[x_1,\ldots,x_d]$. Let $\alg[{\bm x}]_1$ 
be the subspace of linear forms of $\alg[{\bm{x}}]$ and
let $\Mat_{n\times m}(\alg[{\bm x}]_1)$ be 
the space of $m\times n$ matrices of linear forms in
$\alg[{\bm{x}}]$, writing $\Mat_{n}(\alg[{\bm x}]_1)$ if $n=m$. 
If $X \subset \alg^d$, then $\langle X \rangle_K$ is the $K$-span of $X$.  
Let $\varaff{I}$ be the affine variety of a homogeneous ideal 
$I$ of $\alg[{\bm x}]$ in $\alg^{d}$ and
let $\varproj{I}$ be the corresponding projective variety in
$\mathbb{P}^{d-1}(\alg)$. 
See \cite[Chap.\ 1]{CLO15} for standard concepts in algebraic geometry. 

\subsection{Groups from matrices and back}\label{sec:gps-from-tensors}
We briefly review the well-known Baer correspondence \cite{Baer/38}, 
a consequence of MacLane's Universal Coefficients Theorem
\cite[\S11.4]{Rob96}. Since we consider only 
$p$-groups of class 2 and exponent $p$, 
we restrict to odd primes $p$. 

Let $n$ and $d$ be positive integers and let $\ff$ be the field of $p$ elements. Let ${\bm y}=(y_1,\ldots,y_d)$ 
be a vector of independent algebraic variables. Let 
$\tuB({\bm y})\in\Mat_n(\alg[{\bm y}]_1)$ be a skew-symmetric matrix of linear
forms. Write $V=\alg^n$ and $W=\alg^d$ and let
 $W$ have standard basis $(f_1,\ldots,f_d)$.  
Now $\tuB$ defines the alternating map 
\begin{equation}\label{eq:t}
 t:V\times V \longrightarrow W, \quad (v,v') \longmapsto t(v,v')=\sum_{i=1}^d v^{\top}\tuB(f_i)v'f_i.
\end{equation}
Defining an operation $\star$ on $V \times W$ by 
\[
(v,w)\star(v',w')=\left( v+v', w+w' + \frac{1}{2}t(v,v') \right)
\]
yields a group $\GP_{\tuB}(\alg)$.
Clearly, a change of basis for $V$ or $W$ does not change the isomorphism 
class of $\GP_{\tuB}(\alg)$. Write 
 \[
 \tuB({\bm y})=\left(\tuB_{ij}^{(1)}y_1+\ldots + \tuB_{ij}^{(d)}y_d \right)_{ij} \quad \textup{with }\ \ \tuB_{ij}^{(\kappa)}\in\Mat_n(\alg).
 \]
The \emph{adjoint} of $\tuB({\bm y})$ is
 \begin{equation}
 \tuBd({\bm x}) = \left( \sum_{j=1}^n
 \tuB^{(\kappa)}_{ij}x_j\right)_{i\kappa} \in\Mat_{n\times d}(\alg[x_1,\dots,x_n]_1). 
\end{equation}
In the following,  $I_1(\tuB)\subseteq \alg[{\bm y}]$  denotes the ideal generated by the entries of $\tuB$; define analogously $I_1(\tuBd)\subseteq \alg[{\bm x}]$. See also \cref{sec:invs}.
\begin{lemma}\label{lem:I1}
Assume that $\alg=\ff$ and write $G=\GP_{\tuB}(\ff)$. 
\begin{enumerate}[label=$(\arabic*)$]
 \item The order of $G$ is $p^{n+d}$.
 \item The class of $G$ is at most $2$, 
where equality holds if and only if $\tuB$ is not the zero matrix.
 \item The exponent of $G$ is $p$.
 \item\label{it:I1-3} The derived group $G'$ of $G$ is 
an $\ff$-subspace of $\{0 \} \times W$ and has dimension 
 $$\dim_{\ff}G'=d-\dim(\varaff{I_1(\tuB)}).$$
 \item\label{it:I1-4} The centre $\ZG(G)$ of $G$ is 
an $\ff$-subspace of $V\times W$ containing $\{0\}\times W$ and has dimension 
 $$\dim_{\ff}\ZG(G)=d+\dim(\varaff{I_1(\tuBd)}).$$
\end{enumerate}
\end{lemma}

Now let $G$ be a $p$-group 
of class $2$ and exponent $p$. 
Assume that $|G'|=p^d$ and $|G|=p^{n+d}$. 
Clearly 
$V=G/G'$ and $W=G'$
are $\ff$-vector spaces of dimension $n$ and $d$,
respectively. Moreover, the commutator map induces 
an alternating map
$t:V\times V\longrightarrow W$. 
By choosing bases 
of $V$ and
$W$ respectively, we represent $t$ by 
$\tuB({\bm y})\in\Mat_n(\ff[y_1,\ldots,y_d]_1)$,
a skew-symmetric matrix of
linear forms. 
The matrix $\tuBd$ determines the adjoint operator 
on $V$ and its rank loci partition the elements of $G\setminus G'$ 
by centraliser sizes \cite[Lem.~4.4]{StaVoll21}.
If also $G'=\ZG(G)$, then 
$\tuB$ and $\tuBd$ are respectively the matrices
$B$ and $A$ defined in \cite[Def.~2.1]{O'BrienVoll/15}. 

Procedures to construct the Baer-MacLane functor and 
related tensor invariants are part of the {\sc Magma} \cite{Magma}
MultilinearAlgebra package of Maglione and Wilson \cite{FiMaWi19, GITMW}.

\subsection{Varieties and invariants}\label{sec:invs} 
Let $m,n$ and $d$ be positive integers and write 
$N = \min\{m,n\}$. Let ${\bm z}=(z_1,\ldots,z_d)$ be
a vector of independent algebraic variables. 
Let $\tuD({\bm z})\in\Mat_{m\times n}(\alg[{\bm z}]_1)$ be a matrix 
of linear forms. 
\begin{itemize}
 \item For $k\in\{1,\ldots,N\}$, 
the \emph{$k$-th $($determinantal\,$)$ ideal} of $\tuD$ is the homogeneous 
ideal $I_k(\tuD)$ of $\alg[{\bm z}]$ generated by 
the $k\times k$ minors of $\tuD({\bm z})$, while the \emph{$k$-th $($determinantal\,$)$ variety} of $\tuD$ is $V_k=\varaff{I_k(\tuD)}$.
 \item The \emph{ideal rank vector} of $\tuD$ 
is $\irk(\tuD)=(\,I_k(\tuD): k=1,\ldots,N\,)$.
\end{itemize}
 Each $V_i$ is an affine variety in $\alg^d$ and
$\{0\}\subseteq V_1\subseteq V_2\subseteq \ldots \subseteq V_N\subseteq \alg^d$.
See \cite{BV88} for a discussion of determinantal 
ideals and associated varieties. 
\begin{remark}\label{rmk:degree}\label{rmk:primary}
 Let $I$ be a homogeneous ideal in $\alg[x_1,\ldots,x_d]$ and let $m$
be the dimension of $\varproj{I}$. Following
 \cite[Defs.~12.11 and 12.14]{Gath14}, the \emph{degree} of $I$ is 
 the product of $m!$ by the leading coefficient of the Hilbert
 polynomial 
of $I$. If $\varproj{I}$ is empty, then the
 degree of $I$ is $0$. 
The Hilbert polynomial is invariant under change of coordinates
 \cite[Rmk.~12.2]{Gath14}, thus so is the degree of $I$.
Now $I$ is \emph{primary} if, whenever 
$ab\in I$ for $a, b \in R$, 
then $a\in I$, or there exists a positive integer $i$ such that $b^i\in I$.
Also $I$ has 
a \emph{minimal
primary decomposition} in the sense of \cite[Def.~8.20]{Gath13}. Two
minimal decompositions have the same cardinality, 
and the irreducible components
$\varproj{I_1},\ldots,\varproj{I_r}$ do not depend on the chosen
decomposition \cite[Prop.~8.27]{Gath13}. 
For a more detailed treatment,  
see \cite[Chap.\ 3]{Eis95}, 
\cite[\S8]{Gath13}, and \cite[\S12]{Gath14}. 
\end{remark}

\begin{thm}\label{th:invariants}
Let $n$ and $d$ be positive integers, and let ${\bm y}=(y_1,\ldots,y_d)$ and
${\bm x}=(x_1,\ldots,x_n)$ be vectors of independent algebraic
variables. Let $\tuB({\bm y})\in\Mat_n(\ff[{\bm y}]_1)$ be
a skew-symmetric matrix of linear forms and let $\tuBd({\bm
 x})\in\Mat_{n\times d}(\ff[{\bm x}]_1)$ be its adjoint. Write
\begin{itemize}
 \item $\irk(\tuB)=(I_1,\ldots,I_n)$ for the ideal rank vector of $\tuB$, and
 \item $\irk(\tuBd)=(J_1,\ldots,J_{\min\{n,d\}})$ for the ideal
rank vector of $\tuBd$. 
\end{itemize}
For each $k\in\{1,\ldots,n\}$ and $\ell\in\{1,\ldots,\min\{n,d\}\}$, 
the following are invariants of the isomorphism class of $\GP_{\tuB}(\ff)$: 
\begin{enumerate}[label=$(\arabic*)$]
 \item\label{it:inv1} the degrees and dimensions of $I_k$ and $J_{\ell}$;
 \item\label{it:inv2} the degrees and dimensions of the ideals in a primary decomposition of $I_k$ and $J_{\ell}$;
 \item\label{it:inv3} the degrees, the number of rational points, and the number of irreducible components of $\varaff{I_k}$ and $\varaff{J_{\ell}}$;
 \item\label{it:inv4} the dimensions of $\gen{\varaff{I_k}}_{\F_p}$ and $\gen{\varaff{J_{\ell}}}_{\F_p}$.
\end{enumerate}
\end{thm}

\begin{proof}
Let $\tuC\in\Mat_n(\ff[{\bm y}]_1)$ be such that $\GP_{\tuB}(\ff)$ and
$\GP_{\tuC}(\ff)$ are isomorphic, and let $X$ and $Z$ be such that 
$X\tuB({\bm y}Z)X^{\top}=\tuC({\bm y})$. The existence of $X$ and $Z$ is 
guaranteed by the Baer-MacLane construction 
(see \cite[Thm.~2.13(2)]{MaSta22} for a similar formulation).

Fix $k\in\{1,\ldots,n\}$. 
If $v\in\alg^{d}$, then  
$\rk(\tuB(v))=\rk(X\tuB(Zv)X^\top)=\rk(\tuC(v))$. 
Hence $Z$ induces a linear isomorphism of varieties
$\varphi_Z:\varaff{I_k(\tuB)}\rightarrow\varaff{I_k(\tuC)}$.
Being invariants under isomorphism of varieties, the
quantities from \ref{it:inv3} are preserved. 
Since $\varphi_Z$ is linear, 
$\dim_{\ff
}\gen{\varaff{I_k(\tuB)}}_{\F_p}=\dim_{\ff}\gen{\varaff{I_k(\tuC)}}_{\F_p}$,
thus \ref{it:inv4} holds.
Since $\varphi_Z$ is an isomorphism, 
the dimensions of the ideals $I_k(\tuB)$ and $I_k(\tuC)$ equal those 
of the corresponding varieties $\varaff{I_k(\tuB)}$ and
$\varaff{I_k(\tuC)}$. 
Moreover, the degrees of $I_k(\tuB)$
and $I_k(\tuC)$ coincide by \cref{rmk:degree} and the linearity of $Z$. 
Thus \ref{it:inv1} holds. 
Similarly, we deduce \ref{it:inv2}.
The proof for adjoints is analogous.
\end{proof}

\section{An illustrative example} \label{sec:4-3}
Vaughan-Lee \cite[\S7.4]{VL15} lists the six isomorphism 
classes of 4-generator groups of order $p^7$, class 2, and exponent $p$. 
Following \cref{sec:gps-from-tensors}, 
representatives of these isomorphism classes can be constructed over 
$\F_p$ (with primitive element $\omega$)
from the following matrices.

\begin{alignat*}{3}
\tuB_1(x,y,z)&=\begin{pmatrix}
 0 & x & y & z \\ -x & 0 & 0 & 0 \\ -y & 0 & 0 & 0 \\ -z & 0 & 0 & 0
\end{pmatrix}, \quad &&\tuB_2(x,y,z)&&=\begin{pmatrix}
 0 & x & y & 0 \\ -x & 0 & z & 0 \\ -y & -z & 0 & 0 \\ 0 & 0 & 0 & 0
\end{pmatrix},\\
\tuB_3(x,y,z)&=\begin{pmatrix}
 0 & x & y & 0 \\ -x & 0 & 0 & z \\ -y & 0 & 0 & 0 \\ 0 & -z & 0 & 0 
\end{pmatrix}, \quad &&\tuB_4(x,y,z)&&=\begin{pmatrix}
 0 & x & y & z \\ -x & 0 & z & 0 \\ -y & -z & 0 & 0 \\
 -z & 0 & 0 & 0 
\end{pmatrix}, \\
\tuB_5(x,y,z)&=\begin{pmatrix}
 0 & x & 0 & y \\ -x & 0 & y & 0 \\ 0 & -y & 0 & z \\
 -y & 0 & -z & 0
\end{pmatrix}, \quad &&\tuB_6(x,y,z)&&=\begin{pmatrix}
 0 & x & y & z \\ -x & 0 & 0 & \omega y \\ -y & 0 & 0 & x \\ -z & -\omega y & -x & 0
\end{pmatrix}.
\end{alignat*}
Since $\dim\varaff{I_1(\tuB)} = 0$ for 
$\tuB\in\{\tuB_1,\ldots,\tuB_6\}$, 
\cref{lem:I1} confirms that these define groups of the claimed type.
The adjoint matrices of $\tuB_1,\ldots,\tuB_6$ are:
 \begin{alignat*}{3}
\tuB_1^\bullet(x,y,z,w)&=\begin{pmatrix}
 y & z & w \\ -x & 0 & 0 \\ 0 & -x & 0 \\ 0 & 0 & -x
\end{pmatrix}, \quad &&\tuB_2^\bullet(x,y,z,w)&&=\begin{pmatrix}
 y & z & 0 \\ -x & 0 & z \\ -0 & -x & -y \\ 
 0 & 0 & 0
\end{pmatrix},
\end{alignat*}
 \begin{alignat*}{3}
\tuB_3^\bullet(x,y,z,w)&=\begin{pmatrix}
 y & z & 0 \\ -x & 0 & w \\ 0 & -x & 0 \\ 0 & 0 & -y 
\end{pmatrix}, \quad &&\tuB_4^\bullet(x,y,z,w)&&=\begin{pmatrix}
 y & z & w \\ -x & 0 & z \\ 0 & -x & -y \\
 0 & 0 & -x 
\end{pmatrix}, \\
\end{alignat*}
 \begin{alignat*}{3}
\tuB_5^\bullet(x,y,z,w)&=\begin{pmatrix}
 y & w & 0 \\ -x & z & 0 \\ 0 & -y & w \\
 0 & -x & -z 
\end{pmatrix}, \quad &&\tuB_6^\bullet(x,y,z,w)&&=\begin{pmatrix}
 y & z & w \\ -x & \omega w & 0 \\ w & -x & 0 \\ 
 -z & -\omega y & -x
\end{pmatrix}.
\end{alignat*}
Let $n_p(I)$ denote the number of $\ff$-rational points of $\varaff{I}$. 
The invariants summarised in Table \ref{4-3-invariants} 
separate the groups up to isomorphism.

\begin{table}[htb]
\begin{center}
{
\begin{tabular}{r|r|r|r|r|r|r}
 $\tuB$ & $\tuB_1$ & $\tuB_2$ & $\tuB_3$ & $\tuB_4$ & $\tuB_5$ & $\tuB_6$ \\ \hline 
 $n_p(I_4(\tuB))$ & $p^3$ & $p^3$ & $2p^2-p$ & $p^2$ & $p^2$ & $p$ \\ \hline
 $n_p(I_3(\tuBd))$ & $p^3$ & $p^4$ & $2p^3-p^2$ & $p^3$ & $p^3+p^2-p$ & $p^2$
\end{tabular}
}
\end{center}
\caption{Invariants of $\tuB_1,\ldots,\tuB_6$}
\label{4-3-invariants}
\end{table}

\section{The $5$-generator $p$-groups of class $2$ and exponent $p$}\label{sec:5-gen}
Brahana \cite{Brah51,Brah58} used concepts from projective geometry
to study some of the 5-generator groups of class 2 and exponent $p$:  
in particular, commutators in such a group 
arise as solutions of a system of homogeneous quadratic equations 
called Pl\"ucker relations; cf.\ \cite[\S4.2]{Cop04}. 
As in other treatments, his construction of 
the groups up to isomorphism reduces to an orbit problem.  
Copetti \cite{Cop04} presents a modern and highly accessible 
treatment of his work.  
Vaughan-Lee \cite{VL15} used an approach similar to 
the $p$-group generation algorithm \cite{OBrien90} to 
determine the groups of order 
dividing $p^8$ having class $2$ and exponent $p$.

A $5$-generator $p$-group of class $2$ and exponent $p$
has order $p^{5 + d}$ where $1 \leq d \leq 10$.
Recall from the introduction that 
we consider the action of $\GL(V)$, where $V = \F_p^5$,
on the exterior square $\Lambda^2 (V)$, a space of dimension $10$.
Two groups of order $p^{5+d}$ are isomorphic
if and only if the duals of their associated $d$-dimensional spaces
of $\Lambda^2(V)$ are in the same $\GL(V)$-orbit. 
For $V = \F_p^n$, there are $\lfloor n / 2 \rfloor$ orbits of 
1-dimensional spaces on $\Lambda^2(V)$: each orbit is determined 
by the rank of the associated form~\cite[Thm.~3.5]{Cop04}.
Brooksbank, Maglione, and Wilson \cite{BMW17} give
a polynomial-time algorithm to decide isomorphism 
between two such groups of order $p^{n + 2}$.
Hence, by exploiting well-known properties
of duality (cf. for example \cite[Chap.\ 2]{Cop04}), 
we consider only $d \in \{3,4, 5\}$. 

\subsection{Derived groups of order $p^3$}
Copetti \cite{Cop04} and Vaughan-Lee \cite{VL15} showed that 
there are $22$ isomorphism classes of 5-generator groups of order 
$p^8$, class 2, and exponent $p$. 
Case ($j$) of \cite[Thm.~7.22]{Cop04} 
corresponds via the Lazard Correspondence to 
Group 8.5.$j$ of \cite[\S8.5]{VL15}.

The matrices $\tuB\in\Mat_5(\ff[y_1,y_2,y_3]_1)$, 
extracted from \cite[\S8.5]{VL15} and described 
as in \cref{sec:gps-from-tensors}, are available at \cite{GITPAGE}.
Since $\tuB$ is skew-symmetric, 
$I_5(\tuB)$ is necessarily the zero ideal, and 
no information on the isomorphism class of $\GP_{\tuB}(\ff)$ is gained 
from the rank $5$ locus of B. 

For $p \in \{3,\ldots,37\}$,
Table \ref{5-3-invariants} records the values of invariants 
from \cref{th:invariants} for these 22 groups. 
There, $\deg(I)$ denotes the degree of an ideal $I$ and 
$\degprim(I)$ records 
the degrees of the ideals in a 
minimal primary decomposition of $I$. 
By comparing Tables \ref{4-3-invariants} and \ref{5-3-invariants}, 
we observe that groups (1)--(6) are direct products 
of the six groups from \cref{sec:4-3} with a group of order~$p$.
Twenty of the $22$ isomorphism classes are identified 
using just three of the invariants
associated with the rank $4$ locus of $\tuB$ and 
the rank $3$ locus of $\tuBd$. 
The invariants from \cref{th:invariants} do 
not distinguish between (14) and (15).  
Yet, Vaughan-Lee \cite[\S8.5]{VL15} showed that 
the automorphism groups of $\GP_{\tuB_{14}}(\ff)$ 
and $\GP_{\tuB_{15}}(\ff)$ have orders 
$2(p-1)(p^2-1)p^{17}$ and $2(p-1)^3p^{17}$ respectively.

\begin{table}[ht]
\begin{center}
{
\begin{tabular}{c|r|c|r|r}
 Case & $n_p(I_4(\tuB))$ & $\deg(I_4(\tuB))$ & $n_p(I_3(\tuBd))$ & $\degprim(I_3(\tuBd))$ \\ \hline 
 (1) & $p^3$ & $1$ & $p^4$ & $(2,3)$ \\ \hline
 (2) & $p^3$ & $1$ & $p^5$ & $(1)$ \\ \hline
 (3) & $2p^2-p$ & $4$ & $2p^4-p^3$ & $(1,1,4)$ \\ \hline
 (4) & $p^2$ & $4$ & $p^4$ & $(2,3)$ \\ \hline
 (5) & $p^2$ & $4$ & $p^4+p^3-p^2$ & $(2,4)$ \\ \hline
 (6) & $p$ & $4$ & $p^3$ & $(2,3)$ \\ \hline
 (7) & $1$ & $9$ & $p^3$ & $(1,3)$ \\ \hline
 (8) & $1$ & $6$ & $p^3-p^2+p$ & $(2,2)$ \\ \hline
 (9) & $1$ & $0$ & $p^3$ & $(4)$ \\ \hline
 (10) & $p$ & $10$ & $p^3$ & $(4)$ \\ \hline
 (11) & $p$ & $9$ & $2p^3-p^2$ & $(1,3)$ \\ \hline
 (12) & $p$ & $6$ & $2p^3-p^2$ & $(2,2)$ \\ \hline
 (13) & $p$ & $9$ & $2p^3-p^2$ & $(1,1,2)$ \\ \hhline{|=|=|=|=|=|}
 (14) & $p$ & $3$ & $2p^3-p$ & $(1,3)$ \\ \hline
 (15) & $p$ & $3$ & $2p^3-p$ & $(1,3)$ \\ \hhline{|=|=|=|=|=|}
 (16) & $2p-1$ & $9$ & $3p^3-2p^2$ & $(1,1,2)$ \\ \hline
 (17) & $2p-1$ & $6$ & $3p^3-p^2-p$ & $(1,1,2)$ \\ \hline
 (18) & $3p-2$ & $9$ & $4p^3-3p^2$ & $(1,1,1,1)$ \\ \hline
 (19) & $p^2+p-1$ & $2$ & $p^4+p^3-p^2$ & $(1,1,2)$ \\ \hline
 (20) & $p^2$ & $2$ & $p^4$ & $(1,2,4)$ \\ \hline
 (21) & $p^2$ & $2$ & $p^4+p^3-p^2$ & $(1,1,6)$ \\ \hline
 (22) & $p^2$ & $2$ & $p^4+p^3-p^2$ & $(1,1,2)$ 
\end{tabular}
}
\end{center}
\caption{Invariants for $5$-generator groups with derived group of order $p^3$}\label{5-3-invariants}
\end{table}

\subsection{Quasi-polynomiality} 
Motivated by Higman's PORC conjecture, 
Lee \cite{Lee16} constructed a family $\Lee$ of 
groups of order $p^9$ for all primes $p \geq 5$, 
and showed that its cardinality 
is not determined by a quasi-polynomial function in $p$. 
We briefly review his construction.
Let $\tuB \in\Mat_5(\Z[y_1,y_2,y_3])$  where 
\[
\tuB =\begin{pmatrix}
 0 & 0 & 0 & y_1 & y_2 \\
 0 & 0 & 0 & y_3 & y_1 \\
 0 & 0 & 0 & 2y_2 & y_3 \\
 -y_1 & -y_3 & -2y_2 & 0 & 0 \\
 -y_2 & -y_1 & -y_3 & 0 & 0
\end{pmatrix}.
\]
The group $\GP_{\tuB}(\F_p)$ has order $p^{5+3}=p^8$. 
Its isomorphism type is one of (7), (13) and (18);  
their projective varieties $\varproj{I_4(\tuB)}$ 
have respectively $0$, $1$, and $3$ points in $\P^2(\F_p)$. 
The isomorphism type depends on $p$, each occurs infinitely often, 
and the function
\begin{equation}\label{eq:aut-p}
\mathcal{P}=\{\textup{primes}\}\longrightarrow \N, \quad p\longmapsto |\Aut(\GP_{\tuB}(\F_p))|
\end{equation}
is {\em not} quasi-polynomial. 
The groups in $\Lee$ are \emph{immediate descendants} 
\cite{OBrien90} of $\GP_{\tuB}(\F_p)$: 
they have order $p^{9}$, class $3$, and exponent $p$. 
They are determined by orbit representatives 
for the action of $\Aut(\GP_{\tuB}(\F_p))$ on the duals of 1-spaces in 
$\F_p^{13}$; the number
of orbits, or equivalently the cardinality of $\Lee$, varies 
according to $|\Aut(\GP_{\tuB}(\F_p))|$, 
and so it is also not quasi-polynomial. 
We show that the groups in $\Lee$ have minimal order. 

\begin{thm}
 Assume that 
$n+d\leq 7$ and let $\tuB\in\Mat_n(\Z[y_1,\ldots,y_d]_1)$ be skew-symmetric. 
The function 
$\cor{P}\longrightarrow \N$, defined by $p\mapsto |\Aut(\GP_{\tuB}(\F_p))|$,
is quasi-polynomial.
\end{thm}

\begin{proof}
Assume that $d\leq 2$. Let  $G=\GP_{\tuB}(\Q)$ and 
$G_p=\GP_{\tuB}(\F_p)$. If $\ZG(G)\neq G'$, then we replace $n$ by 
a smaller value, 
so we assume without loss of generality that $\ZG(G)=G'$. The number 
of primes $p$ for which $\ZG(G_p)\neq G_p'$ is finite, since 
the reduction modulo 
$p$ of $\varaff{I_1(\tuB)}$ or 
$\varaff{I_1(\tuBd)}$ now has positive 
dimension (instead of $0$); cf.\ \cref{lem:I1}. 
These primes do not impact on quasi-polynomiality.  
By requiring that 
$\ZG(G_p)= G_p'$ and comparing with \cite{VL15},
we deduce that, for $(n,d)\neq (4,2), (5,2)$, the function 
$p\mapsto |\Aut(\GP_{\tuB}(\F_p))|$ is constant. 
For the two remaining cases,
the list in \cite{VL15} demonstrates that variation over 
the primes is determined either by the dimension of 
$\varaff{I_1(\tuB)}$, or by a given integer being a 
square modulo a prime. Both conditions are quasi-polynomial.

We now assume that $d \geq 3$, 
so the only possible parameters are $(n,d)\in\{(3,3),(4,3)\}$.
If $(n,d)=(3,3)$ then \eqref{eq:aut-p} is quasi-polynomial.  
Consider $(n,d)=(4,3)$. Now $I_4(\tuB)$ is principal, generated by 
the square of a homogeneous quadratic polynomial $f\in\Z[y_1,y_2,y_3]$. 
Using the notation of \cref{sec:4-3}, the projective varieties given by 
the zero locus of $f$ are the following:
\begin{itemize}
\item $\varproj{I_4(\tuB_1)}$ and $\varproj{I_4(\tuB_2)}$ equal $\P^2{(\Q)}$;
 \item $\varproj{I_4(\tuB_3)}$ is a union of two lines in $\P^2{(\Q)}$;
 \item $\varproj{I_4(\tuB_4)}$ is a line in $\P^2{(\Q)}$;
 \item $\varproj{I_4(\tuB_5)}$ is an ellipse in $\P^2{(\Q)}$;
 \item $\varproj{I_4(\tuB_6)}$ is a point in $\P^2{(\Q)}$.
\end{itemize}
If $f=0$, then $\GP_{\tuB}(\F_p)$ is isomorphic to either 
$\GP_{\tuB_1}(\F_p)$ or $\GP_{\tuB_2}(\F_p)$. 
By \cite[\S7.4]{VL15}, their automorphism groups have the same order, 
and so \eqref{eq:aut-p} is quasi-polynomial. 
Assume that $f\neq 0$.
Since the coefficients of $f$ have a finite number of common prime divisors, 
we can ignore $\tuB_1$ and $\tuB_2$. 
Assume that $\GP_{\tuB}(\F_p)\cong\GP_{\tuB_5}(\F_p)$ for some prime $p$. 
Now $f$ is necessarily smooth 
over $\Z$, and there are only finitely many primes $r$ 
(of bad reduction) such that $\varproj{I_4(\tuB)}$ is a point, a line, 
or a union of two lines in $\P^2(\F_r)$. By multiplying the bad primes $r$ 
into the modulus, and using the automorphism group orders 
from \cite[\S7.4]{VL15}, we deduce that 
\eqref{eq:aut-p} is quasi-polynomial. 
Consider the splitting behaviour of $f$ in the three remaining cases. 
Let $r$ be an odd prime. One of the following holds:
\begin{enumerate}[start=0]
\item\label{it:f0} $\varproj{I_4(\tuB)}$ is a point 
in $\P^2(\F_r)$ if and only if $f\bmod r$ is (equivalent to) an 
irreducible polynomial in two variables;
\item\label{it:f1} $\varproj{I_4(\tuB)}$ is a line in 
$\P^2(\F_r)$ if and only if $f\bmod r$ is (a scalar multiple of) a square;
\item\label{it:f2} $\varproj{I_4(\tuB)}$ is a union of 
two lines in $\P^2(\F_r)$ if and only if $f\bmod r$ is reducible 
but not (a scalar multiple of) a square.
\end{enumerate}
If $f$ is a polynomial in three indeterminates, 
then \eqref{it:f0} can occur only for a finite number of primes, and 
so can be ignored.  So either $f$ is a square, as  
is every reduction mod $r$, or \eqref{it:f1} occurs for only 
finitely many primes. In either case, the formulae from 
\cite[\S7.4]{VL15} show that \eqref{eq:aut-p} is quasi-polynomial.  
If $f$ is a polynomial in two indeterminates, 
then the splitting of $f$ modulo $r$ is determined by a quasi-polynomial 
condition, again implying that \eqref{eq:aut-p} is quasi-polynomial.  
\end{proof}

\subsection{Derived groups of order $p^4$}
Brahana \cite{Brah51, Brah58} identified 59 
4-dimensional subspaces of the 10-dimensional commutator space
as orbit representatives.
Copetti \cite{Cop04} showed that the spaces he labelled 29 and 31 are 
in the same orbit.  In unpublished work from 2001, M.F.\ Newman and O'Brien
found two additional redundancies. 
\begin{itemize}
\item Spaces 50 and 51 are in the same orbit.
\item Space 35 is in the 
orbit of 30 when $p = 5$;  
it is in the orbit
of 29 when $p \equiv \pm 1 \bmod 5$;  
and it is in the orbit
of 32 when $p \equiv \pm 2 \bmod 5$. 
\end{itemize}
They also identified one missing orbit representative. 
Independently, Eick and Vaughan-Lee \cite{E-MRVL} showed 
that the number of these groups is 57. 
We summarise the outcome.
\begin{thm}\label{thm:lee}
For every prime $p \geq 3$, there are $57$ isomorphism 
classes of $5$-generator groups of order $p^9$, class $2$, and exponent $p$. 
\end{thm}
The matrices $\tuB\in\Mat_5(\ff[y_1,y_2,y_3]_1)$ for the groups 
determined 
by the duals of the 60 4-dimensional subspaces are available at \cite{GITPAGE}. 
For $p \in \{3,\ldots,37\}$, we investigated the values of 
invariants from \cref{th:invariants} for these groups;
the resulting 57 distinct sets of invariants support 
the amendments to Brahana's list.  
Table \ref{Table5-4} records the values of 
some distinguishing invariants for the 57 cases.  
We retain our earlier notation;
also $\mathrm{rad}(I)$ denotes the radical of an ideal $I$, and 
$\irr_a(I)$ is the number of irreducible components of $\varaff{I}$,
and $\# \prim(I)$ is the number of ideals in a 
minimal primary decomposition of $I$.

We label as (55) the missed representative, otherwise 
we use Brahana's labels \cite{Brah51, Brah58}; 
we omit cases (31), (35), and (51).
For compactness, we write $I_4=I_4(\tuB)$ and $I^\bullet_3=I_3(\tuBd)$.  
For $p=3$ the value $\#\prim(I^\bullet_3)$ for (26) is $3$, not~$4$.

\begin{table}[htp]
\begin{center}
{\small
\begin{tabular}{l|r|c|c|c|r|c|c|c}
 Case & $n_p(I_4)$ & $\deg(\Rad(I_4)$ & $\irr_a(I_4)$ & $\#\prim(I_4)$ & $n_p(I^\bullet_3)$ & $\deg(\Rad(I^\bullet_3))$ & $\irr_a(I^\bullet_3)$ & $\#\prim(I^\bullet_3)$ \\ \hline 
 (0) & $p^4$ & $1$ & $1$ & $1$ & $p^4$ & $1$ & $1$ & $2$ \\ \hline
 (1) & $p^3+p^2-p$ & $2$ & $1$ & $1$ & $2p^3-p$ & $2$ & $2$ & $3$ \\ \hline
 (2) & $p^3-p^2+p$ & $2$ & $1$ & $1$ & $p$ & $2$ & $1$ & $2$ \\ \hline
 (3) & $p^3$ & $2$ & $1$ & $1$ & $p^3$ & $1$ & $1$ & $2$ \\ \hline
 (4) & $2p^2-p$ & $3$ & $2$ & $2$ & $p^3+p^2-p$ & $1$ & $2$ & $4$ \\ \hline
 (5) & $p^2+p-1$ & $2$ & $2$ & $3$ & $p^2+p-1$ & $2$ & $2$ & $3$ \\ \hline
 (6) & $p^2$ & $2$ & $1$ & $3$ & $p^2$ & $2$ & $1$ & $3$ \\ \hline
 (7) & $2p^3-p^2$ & $2$ & $2$ & $2$ & $p^4$ & $1$ & $1$ & $3$ \\ \hline
 (8) & $p^3+p^2-p$ & $1$ & $2$ & $2$ & $p^4$ & $1$ & $1$ & $4$ \\ \hline
 (9) & $p^3$ & $1$ & $1$ & $2$ & $p^4$ & $1$ & $1$ & $3$ \\ \hline
 ($9'$) & $p^3$ & $1$ & $1$ & $2$ & $p^4$ & $1$ & $1$ & $4$ \\ \hline
 (10) & $p^3+p-1$ & $1$ & $2$ & $2$ & $p^3+p^2-1$ & $1$ & $2$ & $2$ \\ \hline
 (11) & $p^3$ & $1$ & $1$ & $2$ & $p^3$ & $1$ & $1$ & $2$ \\ \hline
 (12) & $3p^2-2p$ & $3$ & $3$ & $3$ & $p^3+2p^2-2p$ & $1$ & $3$ & $6$ \\ \hline
 (13) & $3p^2-2p$ & $3$ & $3$ & $4$ & $p^3$ & $1$ & $1$ & $5$ \\ \hline
 (14) & $2p^2-1$ & $2$ & $3$ & $4$ & $4p^2-2p-1$ & $4$ & $4$ & $6$ \\ \hline
 (15) & $2p^2-p$ & $2$ & $2$ & $2$ & $p^3+p^2-p$ & $1$ & $2$ & $4$ \\ \hline
 (16) & $2p^2-p$ & $2$ & $2$ & $4$ & $2p^2-p$ & $2$ & $2$ & $4$ \\ \hline
 (17) & $2p^2-p$ & $2$ & $2$ & $4$ & $3p^2-2p$ & $3$ & $3$ & $5$ \\ \hline
 (18) & $2p^2-p$ & $2$ & $2$ & $3$ & $p^3$ & $1$ & $1$ & $5$ \\ \hline
 (19) & $p^2+2p-2$ & $1$ & $3$ & $4$ & $3p^2+p-3$ & $3$ & $4$ & $5$ \\ \hline
 (20) & $p^2+p-1$ & $1$ & $2$ & $3$ & $2p^2+p-2$ & $2$ & $3$ & $4$ \\ \hline
 ($20'$) & $p^2+p-1$ & $1$ & $2$ & $3$ & $2p^2-1$ & $2$ & $2$ & $4$ \\ \hline
 ($20''$) & $p^2+p-1$ & $1$ & $2$ & $4$ & $3p^2-p-1$ & $3$ & $3$ & $6$ \\ \hline
 (21) & $p^2$ & $3$ & $2$ & $3$ & $p^3$ & $1$ & $1$ & $4$ \\ \hline
 ($21'$) & $p^2$ & $1$ & $1$ & $2$ & $p^3$ & $1$ & $1$ & $4$ \\ \hline
 (22) & $p^2$ & $3$ & $2$ & $2$ & $p^3$ & $1$ & $2$ & $4$ \\ \hline
 (23) & $p^2$ & $1$ & $1$ & $3$ & $p^2$ & $1$ & $1$ & $3$ \\ \hline
 (24) & $p^2$ & $1$ & $2$ & $3$ & $p^2+p-1$ & $3$ & $3$ & $4$ \\ \hline
 (25) & $p^2$ & $1$ & $1$ & $3$ & $2p^2-p$ & $2$ & $2$ & $5$ \\ \hline
 (26) & $p^2$ & $1$ & $1$ & $3$ & $2p^2-p$ & $2$ & $2$ & $4$ \\ \hline
 (27) & $p^2$ & $1$ & $1$ & $4$ & $3p^2-2p$ & $3$ & $3$ & $4$ \\ \hline
 (28) & $p^2$ & $3$ & $1$ & $1$ & $p^3$ & $1$ & $1$ & $2$ \\ \hline
 (29) & $5p-4$ & $5$ & $5$ & $5$ & $10p-9$ & $10$ & $10$ & $11$ \\ \hline
 (30) & $4p-3$ & $4$ & $4$ & $4$ & $7p-6$ & $7 $ & $7$ & $8$ \\ \hline
 (32) & $3p-2$ & $5$ & $4$ & $4$ & $4p-3$ & $10$ & $7$ & $8$ \\ \hline
 (33) & $3p-2$ & $3$ & $3$ & $3$ & $4p-3$ & $4$ & $4$ & $5$ \\ \hline
 (34) & $3p-2$ & $3$ & $3$ & $3$ & $5p-4$ & $5$ & $5$ & $6$ \\ \hline
 (36) & $2p-1$ & $5$ & $3$ & $3$ & $p$ & $10$ & $4$ & $5$ \\ \hline
 (37) & $2p-1$ & $4$ & $3$ & $3$ & $3p-2$ & $7$ & $5$ & $6$ \\ \hline
 (38) & $2p-1$ & $2$ & $2$ & $2$ & $3p-2$ & $3$ & $3$ & $4$ \\ \hline
 (39) & $2p-1$ & $2$ & $2$ & $3$ & $2p-1$ & $4$ & $2$ & $4$ \\ \hline
 (40) & $2p-1$ & $2$ & $2$ & $2$ & $2p-1$ & $2$ & $2$ & $3$ \\ \hline
 (41) & $2p-1$ & $2$ & $2$ & $2$ & $p^2+p-1$ & $2$ & $2$ & $5$ \\ \hline
 (42) & $2p-1$ & $2$ & $2$ & $2$ & $p^2+p-1$ & $2$ & $2$ & $4$ \\ \hline
 (43) & $p$ & $3$ & $1$ & $2$ & $p^3$ & $1$ & $1$ & $3$ \\ \hline
 (44) & $p$ & $1$ & $1$ & $1$ & $p^2$ & $2$ & $1$ & $3$ \\ \hline
 (45) & $p$ & $2$ & $1$ & $3$ & $p$ & $2$ & $1$ & $3$ \\ \hline
 (46) & $p$ & $3$ & $2$ & $2$ & $2p-1$ & $4$ & $3$ & $4$ \\ \hline
 (47) & $p$ & $1$ & $1$ & $1$ & $p$ & $1$ & $1$ & $2$ \\ \hline
 (48) & $p$ & $4$ & $2$ & $2$ & $p$ & $7$ & $3$ & $4$ \\ \hline
 (49) & $p$ & $3$ & $2$ & $2$ & $p$ & $5$ & $3$ & $4$ \\ \hline
 (50) & $p$ & $5$ & $3$ & $3$ & $2p-1$ & $10$ & $6$ & $7$ \\ \hline
 (52) & $p$ & $5$ & $2$ & $2$ & $1$ & $10$ & $3$ & $4$ \\ \hline
 (53) & $1$ & $5$ & $2$ & $2$ & $p$ & $10$ & $3$ & $4$ \\ \hline
 (54) & $1$ & $5$ & $1$ & $1$ & $1$ & $10$ & $2$ & $3$ \\ \hline
 (55) & $p^2$ & $1$ & $1$ & $3$ & $p^2$ & $3$ & $2$ & $3$ 
\end{tabular}
} 
\end{center}
\caption{Invariants for 5-generator groups with 
derived group of order $p^4$} \label{Table5-4}
\end{table}

\subsection{Derived groups of order $p^5$}
\label{derived-5}
Eick and Vaughan-Lee \cite{E-MRVL} showed that the number of these groups is 
$63 + 3p + 2\gcd(p, 2) + \gcd(p, 3) + 2 \gcd(p - 1, 3) + \gcd(p - 1, 4)$. 
No listing of isomorphism types is known. 
Brahana \cite{Brah60} considered 11-groups of this type, 
listing eight isomorphism types.
We report just one computation: 
by generating ``random" $5\times 5$ skew-symmetric 
matrices in five indeterminates over $\F_5$ and partitioning 
them using the invariants of \cref{th:invariants},
we identified 60 of the 87 isomorphism types. 

\section{Providing access to our results}\label{sec:magma}
The lists of skew-symmetric matrices 
defining groups of order $p^{4+ 3}, p^{5 + 3}$ and $p^{5+4}$ are 
available at \cite{GITPAGE}.
We provide {\sc Magma} functions to construct from a 
group of class 2 and exponent $p$ the matrices $\tuB$ and ${\tuBd}$
and vice versa; and the invariants listed in \cref{th:invariants}. 
Our machinery relies on Gr\"obner basis computations, 
but typically has good practical performance. 
Using {\sc Magma} 2.28-15 on a 2.6GHz machine, Tables
\ref{5-3-invariants} and \ref{Table5-4}
were constructed in 17 and 124 minutes of CPU time respectively.
The most expensive calculation is 
the primary decomposition of an ideal; it 
can be omitted, as was done in the calculations reported in Section
\ref{derived-5}. 


\begin{thebibliography}{10}

\bibitem{Adian58}
S.~Adian.
\newblock On algorithmic problems in effectively complete classes of groups.
\newblock {\em Doklady Akad. Nauk. SSR}, 123:13--16, 1958.

\bibitem{Baer/38}
R.~Baer.
\newblock Groups with abelian central quotient group.
\newblock {\em Trans. Amer. Math. Soc.}, 44(3):357--386, 1938.

\bibitem{Magma}
W.~Bosma, J.~Cannon, and C.~Playoust.
\newblock The {M}agma algebra system. {I}. {T}he user language.
\newblock {\em J. Symbolic Comput.}, 24(3-4):235--265, 1997.

\bibitem{Brah51}
H.~R. Brahana.
\newblock Finite metabelian groups and the lines of a projective four-space.
\newblock {\em Amer. J. Math.}, 73:539--555, 1951.

\bibitem{Brah58}
H.~R. Brahana.
\newblock Metabelian {$p$}-groups with five generators and orders {$p\sp{12}$} and {$p\sp{11}$}.
\newblock {\em Illinois J. Math.}, 2:641--717, 1958.

\bibitem{Brah60}
H.~R. Brahana.
\newblock On metabelian groups of order $p^{10}$ with five generators.
\newblock {\em Rend.\ Circ.\ Mat.\ Palermo $(2)$}, 9:97--113, 1960.

\bibitem{BMW17}
P.~A. Brooksbank, J.~Maglione, and J.~B. Wilson.
\newblock A fast isomorphism test for groups whose {Lie} algebra has genus 2.
\newblock {\em J. Algebra}, 473:545--590, 2017.

\bibitem{BV88}
W.~Bruns and U.~Vetter.
\newblock {\em Determinantal rings}, volume~45 of {\em Monogr.\ Mat.}
\newblock Instituto de Matem\'atica Pura e Aplicada (IMPA), Rio de Janeiro, 1988.

\bibitem{Cop04}
A.~Copetti.
\newblock Finite-dimensional {L}ie algebras of nilpotency class 2.
\newblock Master's thesis, Australian National University, 2004.
\newblock \url{https://openresearch-repository.anu.edu.au/handle/1885/150852}.

\bibitem{CLO15}
D.~A. Cox, J.~Little, and D.~O'Shea.
\newblock {\em Ideals, varieties, and algorithms. An introduction to computational algebraic geometry and commutative algebra}.
\newblock Undergrad. Texts Math. Springer, Cham, 4th edition, 2015.

\bibitem{Dehn11}
M.~Dehn.
\newblock \"{U}ber unendliche diskontinuierliche gruppen.
\newblock {\em Math.\ Ann.}, 71:116--144, 1911.

\bibitem{E-MRVL}
B.~Eick and M.~R. Vaughan-Lee.
\newblock Counting $p$-groups and {L}ie algebras using {PORC} formulae.
\newblock {\em J. Algebra}, 545:198--212, 2020.

\bibitem{Eis95}
D.~Eisenbud.
\newblock {\em Commutative Algebra. With a view toward Algebraic Geometry}, volume 150 of {\em Grad. Texts in Math.}
\newblock Springer-Verlag, New York, 1995.

\bibitem{FiMaWi19}
U.~{First}, J.~{Maglione}, and J.~B. {Wilson}.
\newblock {A spectral theory for transverse tensor operators}.
\newblock 2019.
\newblock \url{https://arxiv.org/abs/1911.02518}.

\bibitem{Gath13}
A.~Gathmann.
\newblock Commutative algebra.
\newblock {\em Class Notes TU Kaiserslautern}, 2013.
\newblock \url{https://agag-gathmann.math.rptu.de/class/commalg-2013/commalg-2013.pdf}.

\bibitem{Gath14}
A.~Gathmann.
\newblock Algebraic geometry.
\newblock {\em Class Notes TU Kaiserslautern}, 2014.
\newblock \url{https://agag-gathmann.math.rptu.de/class/alggeom-2014/alggeom-2014.pdf}.

\bibitem{Lee16}
S.~Lee.
\newblock A class of descendant {$p$}-groups of order {$p^9$} and {H}igman's {PORC} conjecture.
\newblock {\em J. Algebra}, 468:440--447, 2016.

\bibitem{MaSta22}
J.~Maglione and M.~Stanojkovski.
\newblock Smooth cuboids in group theory.
\newblock {\em Algebra Number Theory}, 19(5):967--1006, 2025.

\bibitem{GITMW}
J.~Maglione and J.~B. Wilson.
\newblock Tensorspace, version 2.3.
\newblock {\em GitHub}, 2020.
\newblock Contributions from Peter A. Brooksbank. \url{https://github.com/thetensor-space/TensorSpace}.

\bibitem{OBrien90}
E.~A. O'Brien.
\newblock The {\it p}-group generation algorithm.
\newblock {\em J.\ Symbolic Comput.}, 9:677--698, 1990.

\bibitem{OBrien94}
E.~A. O'Brien.
\newblock Isomorphism testing for $p$-groups.
\newblock {\em J.\ Symbolic Comput.}, 17:133--147, 1994.

\bibitem{GITPAGE}
E.~A. O'Brien and M.~Stanojkovski.
\newblock Geometric invariants for finite $p$-groups.
\newblock \url{https://github.com/eamonnaobrien/Invariants-class2-pgroups}.

\bibitem{OBS24}
E.~A. {O'Brien} and M.~{Stanojkovski}.
\newblock {Geometric invariants for $p$-groups of class 2 and exponent $p$}, 2024.
\newblock \url{https://arxiv.org/abs/2411.19555}.

\bibitem{O'BrienVoll/15}
E.~A. O'Brien and C.~Voll.
\newblock Enumerating classes and characters of {$p$}-groups.
\newblock {\em Trans. Amer. Math. Soc.}, 367(11):7775--7796, 2015.

\bibitem{Rabin58}
M.~Rabin.
\newblock Recursive unsolvability of group theoretic problems.
\newblock {\em Ann.\ Math.}, 67:172--194, 1958.

\bibitem{Rob96}
D.~J.~S. Robinson.
\newblock {\em A course in the theory of groups}, volume~80 of {\em Grad.\ Texts in Math.}
\newblock Springer-Verlag, New York, second edition, 1996.

\bibitem{StaVoll21}
M.~Stanojkovski and C.~Voll.
\newblock Hessian matrices, automorphisms of {$p$}-groups, and torsion points of elliptic curves.
\newblock {\em Math. Ann.}, 381(1-2):593--629, 2021.

\bibitem{Sun23}
X.~Sun.
\newblock Faster isomorphism testing for $p$-groups of class $2$ and exponent $p$.
\newblock In {\em STOC'23--Proceedings of the 55th Annual ACM Symposium on Theory of Computing}, pages 433--440, 2023.

\bibitem{VL15}
M.~Vaughan-Lee.
\newblock The automorphisms of class two groups of prime exponent, 2015.
\newblock \url{https://arxiv.org/abs/1501.00678}.

\end{thebibliography}

\end{document}